\documentclass[12pt,leqno,twoside]{amsart}
\usepackage[utf8]{inputenc}
\usepackage[colorlinks=true, pdfstartview=FitV, linkcolor=blue, citecolor=blue]{hyperref}
\usepackage{amstext,amsmath,amscd, bezier,indentfirst,amsthm,amsgen,enumerate, geometry}
\usepackage[all,knot,arc,import,poly]{xy}
\usepackage{amsfonts,color}
\usepackage{amssymb}
\usepackage{latexsym}
\usepackage{epsfig}
\usepackage{graphicx}
\usepackage{srcltx} 
\usepackage{enumitem}


\newcommand{\fin}{\hspace*{\fill}$\square$\vspace*{2mm}}

\theoremstyle{plain}
\newtheorem{theorem}{Theorem}[section]
\newtheorem{lemma}[theorem]{Lemma}
\newtheorem{proposition}[theorem]{Proposition}

\theoremstyle{definition}
\newtheorem{definition}[theorem]{Definition}
\theoremstyle{remark}

\def\bC{{\mathbb C}}

\def\bN{{\mathbb N}}
\def\bP{{\mathbb P}}

\def\bR{{\mathbb R}}

\def\max{{\rm max}}

\def\Sing{{\rm Sing}}

\def\const.{{\rm const.}}

\begin{document}
\title[Tangent cones at infinity]{Tangent cones at infinity}

\author{Luis Renato G. Dias}
\address{Instituto de Matemática e Estatística, Universidade Federal de Uberl\^andia, Av. Jo\~ao Naves de \'Avila 2121, 1F-153 - CEP: 38408-100, Uberl\^andia, Brazil.}
\email{lrgdias@ufu.br}

\author{Nilva Rodrigues Ribeiro}
\address{Universidade Federal de Vi\c cosa, Campus Rio Parana\'iba, Km 7, Zona Rural, MG - 230 Rodovi\'ario - CEP: 38810-000, Rio Parana\'iba, Brazil.}
\email{nilva.ribeiro@ufv.br}



\begin{abstract} Let $X\subset\bC^m$ be an unbounded pure $k$-dimensional algebraic set. We define the tangent cones $C_{4, \infty}(X)$  and $C_{5,\infty}(X)$ of $X$ at infinity. We establish some of their properties and relations. We prove that $X$ must be an affine linear subspace of $\bC^m$ provided that $C_{5, \infty}(X)$ has  pure dimension $k$. Also, we study  the relation between the tangent cones at infinity and representations of $X$ outside a compact set as a branched covering. Our results can be seen as versions at infinity of results of Whitney and Stutz.  
\end{abstract}

\maketitle

\section{Introduction}

Let $(V,0) \subset \bC^m$  be an $n$-dimensional analytic set  in some neighborhood of the origin in $\bC^m$.  In \cite{Wh-65, Wh}, Whitney considered six tangent cones $C_i(V,0)$, $i=1, \ldots, 6$, and, since then, their properties and relations have been studied. The precise definitions for the six cones can be seen in \cite[p. 92]{Ch}, \cite{Wh-65} and \cite{Wh}. The cone $C_3(V, 0)$ is also  known as the Zariski tangent cone of $V$. 

The Whitney tangent cones $C_3(V,0), C_4(V, 0)$ and $C_5(V, 0)$ play an important role in various results. For instance, it is well known that $C_3(V, 0)$ determines the set of all projections $\pi$ of $(V, 0)$ to $\bC^n$ that are a branched cover when restricted to a small neighborhood of $0\in V$. In \cite{St}, Stutz showed that if $\dim C_4(V, 0) =n$, then $\pi$ can be chosen such that the critical locus of $\pi$ is exactly the singular set of $V$. Also in \cite{St}, Stutz proved that if $\dim C_5(V, 0)=n+1$, then a suitable  choice of a projection  of $(V, 0)$ to $\bC^{n+1}$ gives a representation of $V$ in terms of an hypersurface in $\bC^{n+1}$; see also  \cite{Ch} and \cite{Jasna}. 

In this paper, we define the Whitney tangent cones at infinity $C_{4,\infty}(X)$ and $C_{5, \infty}(X)$ for an unbounded pure $k$-dimensional algebraic set $X\subset\bC^m$ and we also consider the Whitney tangent cone at infinity $C_{3, \infty}(X)$  defined in \cite{LP} as follows: 

\begin{definition}\label{d:cones} Let $X\subset\bC^m$ be an unbounded pure $k$-dimensional algebraic set. 
\begin{itemize}
    \item[$\bullet$] The Whitney tangent cone $C_{3, \infty}(X)$ of $X$ at infinity is defined  as the set of the points $v\in \bC^m$ for which there exist sequences $(p_j)_{j}$ in $X$ and $(t_j)_{j}$ in $\bC$ such that $\lim_{j\to \infty} \|p_j\| = \infty$ and $\lim_{j\to \infty}  t_jp_j = v$. 

    \item[$\bullet$] The Whitney tangent cone $C_{4, \infty} (X)$ of $X$ at infinity is defined as the set of the points $ v\in \bC^m$
 for which there exists a  sequence $(p_j)_{j} $ in $X\setminus \Sing  X$ and, for each $j$, there is a vector  $ v_j \in T_{p_j} X$ such that    $  \lim_{j\to\infty} \|p_j\| = \infty$   and $ \lim_{j\to \infty} v_j = v $.

    \item[$\bullet$] The Whitney tangent cone $C_{5, \infty}(X)$ of $X$ at infinity is defined as the set of points $v \in\bC^m$ for which there exist  sequences $(p_j)_{j}$ and $(q_j)_{j}$  in  $X$ and $(t_j)_{j}$ in $\bC$ such that  $\lim_{j\to\infty} \|p_j\|=\lim_{j\to\infty} \|q_j\|=\infty$ and $\lim_{j\to\infty} t_j (p_j -q_j)=v$.
\end{itemize}
\end{definition}

By Definition \ref{d:cones}, for $i=3, 4, 5$, the sets $C_{i,\infty}(X) \subset \bC^m$ are closed and they  are complexly homogeneous in the sense that for any $v \in C_{i,\infty}(X)$ and $\lambda \in \bC$ the vectors  $\lambda v$ also belong to $C_{i,\infty}(X)$. 

In section \S\ref{s:2}, we prove basic properties of $C_{4,\infty}(X)$ and $C_{5, \infty}(X)$. We prove that $C_{4,\infty}(X)$ and $C_{5, \infty}(X)$ are algebraic sets and that $C_{3, \infty}(X) \subset C_{4,\infty}(X) \subset C_{5, \infty}(X)$. Moreover, we prove that $\dim  C_{5, \infty}(X) \leq 2k+1$. Then, since $\dim C_{3, \infty}(X) = k$ (see \cite[Prop. 2.10]{DR}), it follows that $k \leq \dim C_{4,\infty}(X)\leq \dim C_{5, \infty}(X) \leq 2k+1$; see Propositions \ref{p:cones}, \ref{p:C5-alg} and \ref{p:C4-alg}. 

In Section \ref{s:3}, we prove that any pure $k$-dimensional algebraic set $X \subset \bC^m$ for which $C_{5, \infty}(X)$ has  pure dimension $k$ must be an affine linear subspace of $\bC^m$; see Theorem \ref{t:dim-Ptg}. 

The aim of Section \ref{s:4} is to give versions at infinity of the above results of Stutz  for complex algebraic set in $\bC^m$ as follows. Let $X\subset\bC^m$ be an unbounded pure $k$-dimensional algebraic set. Let $W\subset  \bC^m$ be a fixed $m-k$ dimensional vector space such that $W \cap C_{3,\infty}(X)=\{ 0\}$ (this always exists by Lemma \ref{p:exitrans}). We may use the coordinates $(x, y) \in V\times W= \bC^m$, where $V$ and $W$ are orthogonal. Let  $\pi: X  \to V$ be the restriction to $X$ of the canonical projection  $(x,y) \in V\times W \to x \in V$. We denote by $B_R$ the  ball in $ \bC^m$ of radius $R$ centered at the origin.  We have:

\begin{theorem}\label{t:St1} Suppose  $ C_{4, \infty }(X) \cap W= \{ 0\}$. Then, there exists a sufficiently large $R$  such that $ \Sing X \setminus \bar B_R = \Sing\, \pi \setminus \bar B_R$. 
\end{theorem}

By a similar argument from the proof of Lemma \ref{p:exitrans}, it follows that $\dim  C_{4, \infty }(X) =k$ forces $ C_{4, \infty }(X) \cap W= \{ 0\}$. 

We denote $y=(y_1, \ldots, y_{m-k})$ and  $W^i:=\{(0,y) \in  \bC^m \mid y_i = 0 \},$ for $i=1, \ldots, m-k$. We denote $\pi_i:  \bC^m=V\times W \to \bC^k \times \bC=\bC^{k+1}$ the projection defined by $\pi_i(x,y)=(x, y_i)$. It follows that $\ker \pi_i = W^i$.  We have: 

\begin{theorem}\label{t:St2} Suppose that  $C_{5,\infty}(X) \cap W^i = \{0\}$. Then $\pi_i(X)$ is an algebraic  hypersurface in $\bC^{k+1}$ and, for  sufficiently large $R$, the restriction  $\pi_i: X \setminus \bar B_R \to \pi_i(X \setminus \bar B_R)$ is a homeomorphism.  
\end{theorem}

By a similar argument from the proof of Lemma \ref{p:exitrans}, it follows that for the hypothesis of Theorem \ref{t:St2} to be satisfied, it is necessary that $\dim C_{5, \infty}(X) \leq k+1$.  
 
\section{Properties of the tangent cones at infinity}\label{s:2}

\begin{proposition}\label{p:cones} Let $X\subset\bC^m$ be an unbounded pure $k$-dimensional algebraic set. Then, $ C_{3, \infty}(X) \subset C_{4,\infty}(X) \subset C_{5, \infty}(X)$.
\end{proposition}
\begin{proof} Firstly, we prove that $C_{3,\infty}(X) \subset C_{4, \infty}(X)$.  

 Let $v \in C_{3, \infty}(X)$. By Definition \ref{d:cones},  $0\in C_{3,\infty}(X) \cap  C_{4, \infty}(X)$. So it is sufficient to consider  $v\ne 0$. 
 
 By Definition \ref{d:cones}, there exist sequences $(p_j)_j$ in $X$ and $(t_j)_j$ in $\bC$  such that $\lim_{j\to\infty} \|p_j\|=\infty $ and $\lim_{j\to\infty} t_j p_j = v$.  It is obvious that $\lim_{j\to\infty} t_j=0.$

Since $X \setminus  \Sing X $ is dense on $X$, there exists a sequence  $(q_j )_j$  in $X \setminus  \Sing X $  such that $\| p_j - q_j\| <\frac{1}{j}$. Then, by triangle inequality, we have  $\|p_j\| \leq \|p_j -q_j\| + \|q_j\| \leq \frac{1}{j} + \|q_j\|.$ Since $\lim_{j\to\infty} \|p_j\|=\infty$, also $\lim_{j\to\infty} \|q_j\| = \infty$. 

Moreover, by triangle inequality, we have: 
\[  \|t_jq_j -v\| \leq \|t_jq_j -t_j p_j\| + \|t_jp_j -v \| \leq  |t_j| \, \|q_j -p_j\| + \|t_jp_j -v \|.\] 

Since both members of the right side converge to $0$, we have $\lim_{j\to\infty} t_jq_j =  v.$ 

Since $q_j \in X\setminus \Sing X$, it follows  by the  Curve Selection Lemma that there exist   analytic curves  $\psi:(0,\epsilon) \to (X\setminus \Sing X)$  and $t:(0,\epsilon) \to \bC$ such that $\lim_{s\to 0} \|\psi(s)\|=\infty$  and $\lim_{s\to 0} t(s)\psi(s) =v$. 

Now we can use the same arguments from the proof of Lemma 2.2 of \cite{LP} and we get an analytic curve $\varphi$ which is a reparametrization of $\psi$ that satisfies: 
\begin{enumerate}
    \item[a)] $\varphi(s) \in X\setminus \Sing X$, for any $s$. 
    \item [b)] $\lim_{s\to 0}\|\varphi(s)\|=\infty.$
    \item[c)] $\lim_{s\to 0} \frac{\varphi(s)}{\|\varphi(s)\|}= - \lim_{s\to 0} \frac{\varphi'(s)}{\|\varphi'(s)\|}= \frac{v}{\|v\|}.$ 
\end{enumerate}

It follows by a), b) and c)  that $-\frac{v}{\|v\|} \in C_{4, \infty}(X)$. Since $ C_{4, \infty}(X)$ is a cone, it follows that  $v\in  C_{4, \infty}(X) $.  

Now, we prove that $C_{4,\infty}(X) \subset C_{5, \infty}(X)$. Let $v \in C_{4,\infty}(X)$. By the definition of $C_{4,\infty}(X)$, there exists a sequence $(p_j)_{j}$ in $X\setminus \Sing X $ and, for each $j$, a vector $v_j \in T_{p_j} X$,  such that $\lim_{j\to\infty} \|p_j\| = \infty$ and $\lim_{j\to\infty} v_j = v.$

It follows from the definition of $T_{p_j} X$ that, for any $j$, there exists $q_j$ in $X\setminus \Sing X $ and $t_j$ in $\bR$ such that $\|p_j - q_j\| \leq \frac{1}{j}$, $t_j \leq \frac{1}{j}$, and 
\begin{equation}\label{eq:C-345-a}
    \|\frac{p_j- q_j}{t_j} - v_j\|\leq \frac{1}{j}.
\end{equation}

It follows by triangle inequality that $\|p_j\| \leq \|p_j -q_j\| + \|q_j\| \leq \frac{1}{j} + \|q_j\|.$ Since $\lim_{j\to\infty} \|p_j\|=\infty$, we have $\lim_{j\to\infty} \|q_j\|=\infty$.  It follows by \eqref{eq:C-345-a} and $\lim_{j\to\infty}  v_j=v$ that 
\[ \lim_{j\to \infty}  \frac{1}{t_j} (p_j- q_j) = v, \]
which shows  that $v\in C_{5,\infty}(X)$. 
\end{proof}

The next two propositions are motivated by the following facts: by Corollary 3.1 of \cite{LP},  $C_{3,\infty}(X)$ is an algebraic set; by Proposition 2.10 of \cite{DR}, $\dim C_{3,\infty}(X) = \dim X$. 

We denote the complex projective space of dimension $m$ by $\bP^m$. We identify $\bC^m$  with the open set $\{[x_0: x_1: \ldots :x_m] \in \bP^m \mid x_0\neq 0 \}$ of $\bP^m$ through the map $\tau:\bC^m \to \bP^m$, $\tau(x_1, \ldots, x_m)=[1:x_1:\ldots:x_m]$. The hyperplane at infinity is defined by $H_{\infty}:=\{ [x_0: x_1: \ldots : x_m] \in \bP^m \mid x_0=0 \} $.

\begin{proposition}\label{p:C5-alg} Let $X\subset  \bC^m$ be an unbounded pure $k$-dimensional algebraic set. Then, $C_{5, \infty}(X)$ is an algebraic set of dimension less than or equal to $2k+1$. 
\end{proposition}
\begin{proof}  
Let  $A=\{ (x, y, v) \in X\times X \times \bC^m \mid (x-y) \wedge v =0 \}$.  We have that  $A$ is an algebraic set of $\bC^m\times \bC^m \times \bC^m$. We consider $\sigma: A \to X\times X$, defined by $\sigma(x, y, v)=(x, y)$.
	
We denote by $\Delta \subset X\times X $ the subset of $X\times X$ such  that $(x,y) \in \Delta$ if and only if $x=y $. For any $(x,y) \notin \Delta$, the fiber $\sigma^{-1}(x,y)$ has dimension equal to $1$ (in fact, $\sigma^{-1}(x,y)$ is the set of the vectors that are complex proportional to $x -y$). Therefore, $\dim A= 2q +1$.  
 
We consider $A$ and $\sigma^{-1}(\Delta)$ as subsets  of $ \bC^m\times  \bC^m \times  \bC^m \subset \bP^m \times \bP^m \times  \bC^m$. Let $\mathcal{A}:= \overline{A\setminus \sigma^{-1}(\Delta)}$ be the closure of $A\setminus \sigma^{-1}(\Delta)$ in $\bP^m \times\bP^m \times \bC^m$, which is an algebraic set. We have $\dim \mathcal{A}= 2k+1$. 
	
Let $Z:= \mathcal{A} \cap (H_{\infty} \times H_{\infty} \times  \bC^m)$ and   $\tau: Z \to  \bC^m$ defined by $\tau(x,y,v)=v$. It follows that $Z$ is an algebraic set and by definition of $C_{5, \infty}(X)$, we have  $\tau(Z)=C_{5, \infty}(X)$. This shows that $C_{5, \infty}(X)$ is a constructible set. The proof  that $C_{5, \infty}(X)$ is closed in $\bC^m$ is straightforward (diagonal argument). Therefore, $C_{5, \infty}(X)$ is an algebraic set. 

Since $\dim \mathcal{A}= 2k+1$, it follows that $\dim Z \leq 2k+1$. Therefore, the dimension of $\tau(Z)=C_{5, \infty}(X)$ is less than or equal to  $2k+1$.
\end{proof}

\begin{proposition}\label{p:C4-alg} Let $X\subset  \bC^m$ be an unbounded pure $k$-dimensional algebraic set. Then, $C_{4, \infty}(X)$ is an algebraic set and $k \leq \dim C_{4, \infty}(X)\leq 2k+1$.  
\end{proposition}
\begin{proof} Let $\mathcal{I}(X)$ be the vanishing ideal of $X$ and let $\{f_1, \ldots, f_l\}$ be a set of generators of $\mathcal{I}(X)$. We consider $F=(f_1, \ldots, f_l):\bC^m \to \bC^l$. For any $p\in X \setminus \Sing X$, we have $\dim \ker d_p F = k $ and $\ker d_pF=T_pX$, where $d_p F$ denotes the differential of $F$ at $p$; see section 2 of \cite{Mil}. 

Let $A=\{(p,v) \in X \times \bC^m \mid v \in \ker d_pF \} $. Then $A$ is an algebraic set of $\bC^m\times \bC^m$. Since $\Sing X$ is an algebraic set, it follows that $\Sing X \times \bC^m$ is an algebraic set of $\bC^m \times \bC^m$.  We consider $A$ and $\Sing X \times \bC^m$ as subsets of $ \bC^m\times  \bC^m  \subset \bP^m \times  \bC^m$. Let $\mathcal{B}:= \overline{A\setminus (\Sing X \times \bC^m})$ be the closure of $A\setminus (\Sing X \times \bC^m)$ in $\bP^m \times  \bC^m$, which is an algebraic set.  
	
Let $Z:= \mathcal{B} \cap (H_{\infty} \times \bC^m)$ and  let $\tau: Z \to  \bC^m$ be the restriction to $Z$ of the projection $(x, v) \in \bP^m \times \bC^m \to v \in \bC^m$. Then  $Z$ is an algebraic set and, by the definition of $C_{4, \infty}(X)$, we have  $\tau(Z)=C_{4, \infty}(X)$. This shows that $C_{4, \infty}(X)$ is a constructible set. The proof that $C_{4, \infty}(X)$ is closed in $\bC^m$ is straightforward (diagonal argument). Therefore, $C_{4, \infty}(X)$ is an algebraic set. 

By Proposition 2.10 of \cite{DR}, $\dim C_{3,\infty}(X) = \dim X$. By Proposition \ref{p:C5-alg}, $\dim C_{5,\infty}(X) \leq 2k+1$. It follows by Proposition \ref{p:cones} that  $k \leq \dim C_{4, \infty}(X)\leq 2k+1$, which completes the proof. 
\end{proof}

The above results can be seen as versions at infinity of results of Whitney on tangent cones at a point of an analytic variety; see \cite{Ch},  \cite{Wh-65} and \cite{Wh}.  

\section{On the dimension of $\dim C_{5,\infty}(X)$}\label{s:3}

The main result of this section is Theorem \ref{t:dim-Ptg}. We begin by recalling the definition of an algebraic region.  

\begin{definition}[{See p. 672 of \cite{Ru}}]\label{d:Ru} Let $m\geq 2$.  A set $\Omega \subset \bC^m$ is called an algebraic region of type $(k,m)$ if there exist vector subspaces $V_1, V_2$ in $ \bC^m$ and positive real numbers $A, B$ such that the following conditions hold: $\dim V_1=k$, $\dim V_2=m-k$, $ \bC^m = V_1 \oplus V_2$ and $\Omega$ consists of   those $z\in  \bC^m$ for which:     
\begin{equation}\label{eq:alg.reg.}
	 \|z''\| \leq A (1 + \|z'\|)^B,
\end{equation}
	where $z=z' + z''$, with $z'\in V_1$ and $z'' \in V_2$. 
\end{definition} 
 
We have the following geometric property for algebraic sets:  

\begin{theorem}[See Theorem 2 of \cite{Ru}]\label{t:Ru} A closed analytic subset  $V\subset  \bC^m$ of pure dimension $k$  is  algebraic  if and only if $V$ lies in some algebraic region $\Omega$ of type $(k,m)$. 
\end{theorem}

The next result is motivated by Theorem 2.9 of \cite{DR} and it is useful for the proof of Theorem \ref{t:dim-Ptg}. We remember that by Proposition \ref{p:C5-alg} the cone $C_{5, \infty}(X)$ is an algebraic set. 

\begin{proposition}\label{p:V-CV-same}  Let $X\subset  \bC^m$ be a pure $k$-dimensional algebraic set.  Suppose $C_{5, \infty}(X)$ has a pure dimension $k$. Then there exists an algebraic region $\Omega$ of type $(k, m)$ such that $X$,  $C_{3, \infty}(X)$ and $C_{5,\infty}(X)$  lie in $\Omega$.   
\end{proposition} 
 \begin{proof} We follow the proof of Theorem 2.9 of \cite{DR}. 
 
 It follows by Theorem \ref{t:Ru} that  there exist an algebraic region $\tilde \Omega$ of type $(k, m)$ and   vector spaces $V_1$ and $V_2$ in $ \bC^m$, with $\dim V_1=k$, $\dim V_2=m-k$, $ \bC^m= V_1 \oplus V_2$, and positive real numbers $\tilde A, B$ such that for any $z \in \tilde \Omega$ (hence, in particular, for $z\in  C_{5,\infty}(X)$) we have:  
 	\begin{equation}\label{eq:31}
 		\|z''\| < \tilde  A(1+\|z'\|)^B, 
 	\end{equation}  
 	where $z=z' +z''$, with $z'\in V_1$ and $z'' \in V_2$.  
  
We assume in \eqref{eq:31} that $B\geq 1$ since $(1+t)^s \leq (1+t)$, for any $t\geq 0$ and $0 \leq s\leq 1$. 
 	
We claim that there exists a  positive real number  $R$   such that,  for any $w\in X$, we have: 
\begin{equation}\label{eq:2}
 		\|w''\| < R(1+\|w'\|)^{B}, 
\end{equation}
where $w=w' +w''$, with $w'\in V_1$ and $w'' \in V_2$.     
 	
If the claim is not true, then  there exists a sequence   $(w_j)_{j}$ in $X$ such that $ \| w_j''\| > j (1 + \|w_j'\|)^{B} $ and, up to a subsequence, we may suppose that   $\lim_{j\to\infty}  \frac{ w_j''}{\|w_j''\|} = y_0  \in V_2$.  We have  $\lim_{j\to\infty}\|w_j''\| =\infty$. Since $ \bC^m$ is the direct sum of $V_1$ and $V_2$, $\lim_{j\to\infty} \|w_j\|= \infty $. Then 

\[\frac{1}{j} > \frac{(1 + \|w_j'\|)^{B}}{\| w_j''\|}  \geq \frac{(1 + \|w_j'\|)}{\| w_j''\|} \geq  \frac{\|w_j'\|}{\| w_j''\|} .\] 
 	
Let $t_j := (\tilde A)^{-1}\|w_j''\|$. We have  $\lim_{j\to\infty} t_j =\infty$ and 
 	
\[ \lim_{j\to\infty}\frac{1}{t_j} w_j = \lim_{j\to\infty}\frac{1}{t_j} \left( w_j' + w_j'' \right) = \tilde A y_0.\] 
 
It follows that  $\tilde A y_0\in C_{3, \infty}(X)$ by Definition  \ref{d:cones}. Now, since $V_2$ is a linear space and $ y_0 \in V_2$, we have $\tilde Ay_0 \in V_2 \cap C_{3, \infty}(X)$. By Proposition \ref{p:cones}, it follows that $\tilde Ay_0 \in V_2 \cap C_{5,\infty}(X)$.  This is a contradiction, because $\tilde Ay_0$ does not verify \eqref{eq:31}.  

Therefore, there exists a positive real number  $R$   such that, for any $w\in X$,  
\[ 	\|w''\| < R(1+\|w'\|)^{B}, \]
where $w=w' +w''$, with $w'\in V_1$ and $w'' \in V_2$.   Let $ A:= \max\{R, \tilde A\}$, then $X$ and $C_{5,\infty}(X)$  lie in an algebraic region $\Omega$ of type $(k,m)$, with vector spaces $V_1$ and $V_2$ in $ \bC^m$, with $\dim V_1=k$, $\dim V_2=m-k$  and positive real numbers $A$ and $B$. Since $C_{3,\infty}(X) \subset C_{5,\infty}(X)$, it follows that $C_{3,\infty}(X)$ lies in $\Omega$.  This finishes the proof. 
 \end{proof} 	

In the proof of Theorem \ref{t:dim-Ptg}, we use the definition of the degree of an algebraic set. We follow section 11.3 of \cite{Ch}. As before, the complex projective space is denoted by $\bP^m$ and $ \bC^m$ is identified with the open set $\{[x_0: x_1: \ldots :x_m] \in \bP^m \mid x_0\neq 0 \}$ of $\bP^m$ through the map $\tau: \bC^m \to \bP^m$, $\tau(x_1, \ldots, x_m)=[1:x_1:\ldots:x_m]$.   
 
\begin{definition}\label{d:degree} Let $X\subset  \bC^m$ be an algebraic set. The degree of $X$, denoted by $\deg X$,  is the degree of its closure in $\bP^m$.   
\end{definition}

The next one is the main result of this section:

\begin{theorem}\label{t:dim-Ptg} Let $X\subset  \bC^m$ be a pure $k$-dimensional algebraic set.  Suppose $C_{5, \infty}(X)$ has a pure dimension $k$. Then  $X$ is an affine space. 
\end{theorem}
\begin{proof}  By Proposition \ref{p:V-CV-same}, there exists an algebraic region $\Omega$ of type $(k, m)$ such that $X$, $C_{3, \infty}(X)$ and $C_{5, \infty}(X)$ lie in $\Omega$. Then, there exist  vector spaces $V_1, V_2\subset  \bC^m$ and  constants $A, B$ as in Definition \ref{d:Ru}. We may choose linear coordinates $(x, y)$ in $ \bC^m$ such that $x\in V_1$ and $y\in V_2$. In this coordinate system, we denote by $\pi_{\Omega}$ the canonical projection from  $ \bC^m$ to $V_1$, i.e., $\pi_{\Omega}(x,y)=x$. Moreover, we may assume that $V_1$ and $V_2$ are orthogonal to each other and that $\pi_{\Omega}$ is an orthogonal projection; see Theorem 3 of page 78 of \cite{Ch}. 

We have the following direct fact concerning the projection $\pi_{\Omega}$: 
\begin{lemma}\label{l:1}  The restrictions  $\pi_{\Omega}: X \to V_1$ and $\pi_{\Omega}: C_{5, \infty}(X) \to V_1$ are  proper maps.   
\end{lemma}
\fin 

Then, it follows that the closures of $X$ and $V_2$ (respectively, the closures of $C_{5, \infty}(X)$ and $V_2$) in $\bP^m$ do not have points at infinity in common (here, the phrase ``at infinity'' means the set  $\bP^m \setminus  \bC^m$). Then, it follows by Corollary 1 of \cite[p. 126]{Ch} the following fact:  

\begin{lemma}\label{l:proj} The restriction $\pi_{\Omega}: X \to V_1$ $($respectively, $\pi_{\Omega}: C_{5, \infty} (X) \to V_1$$)$ is a ramified covering  with number of sheets equal to $\deg X$ $($respectively, $\deg C_{5, \infty}(X))$. 	 
\end{lemma}
\fin 

Since $C_{5, \infty}(X)$ is a cone, it follows by Lemma \ref{l:1} that:  
\begin{equation}\label{eq:kernel}
\ker \pi_{\Omega} \cap C_{5, \infty}(X)=\{0\}.
\end{equation}
	
We will show that $\deg X =1$.  Let $\Sigma\subset V_1$ be  the critical values of $\pi_{\Omega} : X \to V_1$ and let $v\notin C_{3,\infty}(\Sigma)$. It follows that $tv \notin \Sigma$, for $t>0$ big enough. 

If $\deg X >1$, then there are two different liftings $\lambda_{1}(t), \lambda_{2}(t)$ of the path $tv$ by $\pi_{\Omega} : X \to V_1$, with $t>0$ big enough. So $\pi_{\Omega}(\lambda_1(t))=\pi_{\Omega}(\lambda_2(t))=tv$. 

 It follows that 
\begin{equation}\label{eq:to0} 
		\left(\lambda_{1}(t) -\lambda_{2}(t)\right) \in \ker \pi_{\Omega}. 
	\end{equation}
 
 We may assume that there is a sequence $(t_j)_j$ in $\bR$ such that $\lim_{j\to\infty}t_j= \infty$ and  
\begin{equation}\label{eq:to1} 
	w:= 	\lim_{j\to \infty}	\frac{\lambda_{1}(t_j) -\lambda_{2}(t_j)}{\|\lambda_{1}(t_j) -\lambda_{2}(t_j)\|}. 
\end{equation}

It follows that $w \in \ker \pi_{\Omega}$. We have $\lim_{j\to\infty}\|\lambda_1(t_j)\|=\lim_{j\to\infty}\|\lambda_2(t_j)\|=
\infty$  since $\pi_{\Omega}$ is an orthogonal projection. Hence   
\[ w \in \ker \pi_{\Omega}   \cap C_{5, \infty}(X).\] 

But this  is a contradiction with \eqref{eq:kernel} since $w\neq 0$. 
	
Thus,  $\deg X=1$. By Proposition 3.3 of \cite{FS}, we have that $X$ is an affine space.  
\end{proof}
 
\section{Tangent cones and projections}\label{s:4}

Let $X\subset \bC^m$ be an unbounded pure $k$-dimensional algebraic set fixed. We begin with the following definition: 

\begin{definition} We say that an $m-k$ dimensional vector space $W\subset  \bC^m$ is transverse to $X$ provided that $W \cap C_{3,\infty}(X) = \{ 0 \}$.
\end{definition} 

The existence of such transverse vector spaces is shown in the next result.

\begin{lemma}\label{p:exitrans}  There exists an $m-k$ dimensional vector space  transverse to $X$.    
\end{lemma} 
 \begin{proof} By Theorem 2.9 of \cite{DR}, there exists an algebraic region $\Omega$ of type $(k, m)$ such that $X$ and  $C_{3, \infty}(X)$ lie in $\Omega$. Then, there exist vector spaces $V_1$ and $V_2$ in $ \bC^m$, with $\dim V_1=k$, $\dim V_2=m-k$, $ \bC^m= V_1 \times V_2$, and positive real numbers $A, B$ such that for any $z \in  \Omega$ (hence, in particular, for $z\in  X \cup C_{3,\infty}(X)$) we have:  
 	\begin{equation}\label{eq:3}
 		\|z''\| <  A(1+\|z'\|)^B, 
 	\end{equation}  
 	where $z=z' +z''$, with $z'\in V_1$ and $z'' \in V_2$.  

Since $C_{3,\infty}(X)$  is  complexly homogeneous  in the sense that for any $v \in C_{3,\infty}(X)$ and $\lambda \in \bC$ the vector $\lambda v$  also belong to $C_{3,\infty}(X)$, it follows that $V_2 \cap C_{3,\infty}(X) =\{ 0\}$. So, $V_2$ is transverse to $X$. 
\end{proof} 

Let $W\subset  \bC^m$ be a fixed $m-k$ dimensional vector space transverse to $X$. This always exists by Lemma \ref{p:exitrans}. We may use the coordinates $(x, y) \in V\times W= \bC^m$, where $V$ and $W$ are orthogonal. 

Let  $\pi: V\times W \to V$ be the canonical projection, $\pi(x,y)=x$. Then: 

\begin{lemma}\label{p:piprop} The restriction $\pi: X \to V$ is a proper map. 
\end{lemma}
\begin{proof} We claim that there exist positive constants $C$ and $R$ such that \[X \subset  \{ (x,y) \in  \bC^m = V\times W \mid \|y\| < C \|x\|\} \cup B_{R}.\]  

Indeed, if the claim is not true, there exists a sequence $\{(x_j, y_j )\}_j$ in $X$  such that $\lim_{j\to\infty} \|(x_j, y_j)\|= \infty$ and $\|y_j\| \geq j \| x_j\|$, for any $j$. So $\lim_{j\to\infty} \frac{x_j}{\|y_j\|} = 0$. Up to a subsequence, we may assume that $\lim_{j\to\infty} \frac{y_j}{\|y_j\|} =y_0$. Then, $y_0 \in W$ and 
\[ \lim_{j\to\infty} \frac{1}{\|y_j\|}(x_j, y_j) =(0, y_0).\]
By Definition \ref{d:cones}, it follows that $(0,y_0) \in C_{3, \infty}(X)$. Hence $(0,y_0) \in C_{3,\infty}(X) \cap W$. This is a contradiction, because $W$ is transverse to $X$. Therefore, the claim holds. 
 
Now, it is a direct consequence of the claim that $\pi$ restricted to $X$ is a proper map.
\end{proof}

It follows by Lemmma \ref{p:piprop} that the closures of $X$ and $W$ in $\bP^m$ do not have points at infinity in common (here, the phrase ``at infinity'' means the set  $\bP^m\setminus  \bC^m$). Then, it follows by Corollary 1 of \cite[p. 126]{Ch} the following fact:  

\begin{lemma}\label{l:proj1} The restriction $\pi: X \to V$  is a ramified covering  with number of sheets equal to $\deg X$. \fin 
\end{lemma}

Now we are ready to give the 
\begin{proof}[Proof of Theorem \ref{t:St1}] The inclusion  $\Sing X \subset \Sing \, \pi$ is always true.

Now, assume that  there is no a positive constant $R$  such that $ \Sing \, \pi \setminus \bar B_R \subset  \Sing X \setminus \bar B_R$. Then, there  exists a sequence $(p_j)_j$ in $\Sing \, \pi  \setminus \Sing X$ such that $\lim_{j\to\infty}\|p_j\| = \infty$.  

Then, we can take $w_j \in T_{p_j} X$, $\|w_j\|= 1$, such that $d_{p_j}\pi(w_j)=0$. Since $\pi$ is a linear map, we have $d_{p_j}\pi(w_j)=\pi(w_j)$.  Therefore, $w_j \in W = \ker \pi$, for any $j$. 

Up to a subsequence, we may assume that $\lim_{j\to\infty}w_j = w_0$.  Then $w_0\in W \cap C_{4, \infty }(X)$ which is a contradiction.  
\end{proof}

We have fixed an $m-k$ plane $W$ transverse to $X$ and coordinates $(x,y) \in V\times W= \bC^m$, where $V$ and $W$ are orthogonal. We denote $y=(y_1, \ldots, y_{m-k})$ and we set $W^i:=\{(0,y) \in  \bC^m \mid y_i = 0 \},$ for $i=1, \ldots, m-k$. We denote $\pi_i:  \bC^m=V\times W \to \bC^k \times \bC=\bC^{k+1}$ the projection defined by $\pi_i(x,y)=(x, y_i)$. So we have $\ker \pi_i = W^i$.  

\begin{proof}[Proof of Theorem \ref{t:St2}] Since $C_{5,\infty}(X) \cap W^i = \{0\}$, it follows by Proposition \ref{p:cones} that   $C_{3,\infty}(X) \cap W^i = C_{4,\infty}(X) \cap W^i = \{ 0\}$. 

The restriction of $\pi_i$ to $X$ is a proper map since $C_{3,\infty}(X) \cap W^i =\{ 0\}$. In particular, $\pi_i :X \to \bC^{k+1}$ is a closed map. It follows that  $\pi_i(X) \subset \bC^{k+1}$ is a closed constructible set and, therefore, $\pi_i(X)$ is an algebraic set of dimension $k$. It follows that $\pi_i(X)$ is a hypersurface in $\bC^{k+1}$. 

{\bf Claim 1}: for sufficiently large $R$, the restriction $\pi_i: X \setminus \bar B_R \to \pi(X \setminus \bar B_R)$ is an injective map. 

Indeed, if the claim does not hold, then there exist sequences $(p_j)_j$ and $(q_j)_j$ in $X$ such that $p_j\neq q_j$, $\lim_{j\to\infty}\|p_j\|=\lim_{j\to\infty}\|q_j\|=\infty$ and $\pi_i
(p_j)=\pi_i(q_j)$.  In particular, $p_j - q_j \in W^i$. Up to a subsequence, we may assume that $\lim_{j\to\infty}\frac{p_j -q_j}{\|p_j -q_j\|}= w_0$.   

It follows that $w_0 \neq 0$. By Definition \ref{d:cones}, $w_0 \in C_{5, \infty }(X)$. Moreover, $w_0 \in W^i$ since  $p_j- q_j \in W^i$ for any $j$.   

Then, $w_0\in W^i \cap C_{5, \infty }(X)$, which is a contradiction. Therefore {\bf Claim 1} holds. 

{\bf Claim 2}: for sufficiently large $R$, the restriction $\pi_i: X \setminus \bar B_R \to \pi(X \setminus \bar B_R)$ is a proper map. 

Suppose that for sufficiently large $n\in \bN$, we have that $\pi_i: X \setminus \bar B_n \to \pi(X \setminus \bar B_n)$ is not proper. Then, since $\pi_i: X \to \bC^{k+1}$ is proper, there exists  a sequence $(p_{jn})_{j}$ in $X\setminus \bar B_n$ such that $\lim_{j\to \infty} \pi_i(p_{jn}) = y_n \in  \pi(X \setminus \bar B_n)$ and $\lim_{j\to \infty} p_{jn} = p_{n}$,  where $p_n \in X$ and $\|p_n\|=n$. Thus, there exists $x_n \in X\setminus \bar B_n$  such that $\pi_i(x_n)=y_n$. In particular, $x_n \neq p_n$. 

Then, up to a subsequence, we may assume that $\lim_{n\to \infty}  \frac{p_n -x_n}{\|p_n -x_n\|}= w$.   It follows that $w \neq 0$. Since $p_n- x_n \in W^i$ for any $n$, we have $w \in W^i$.

By construction,   $\lim_{n\to \infty} \|p_n\|=\lim_{n\to\infty} \|x_n\|=\infty$. Hence, $w \in C_{5, \infty }(X).$
  
Then  $w\in W^i \cap C_{5, \infty }(X)$, which is a contradiction. Therefore {\bf Claim 2} holds. 

It follows directly from {\bf Claim 1} and {\bf Claim 2} that $\pi_i: X \setminus \bar B_R \to \pi(X \setminus \bar B_R)$ is a homeomorphism for sufficiently large $R$. 
\end{proof}

\section*{Acknowledgments} 

L. R. G. Dias was partially supported by  Fapemig-Brazil Grant APQ-02085-21 and by Grants  301631/2022-0 and 403959/2023-3 of the Conselho Nacional de Desenvolvimento Cient\'ifico e Tecnol\'ogico (CNPq) of the Ministry of Science, Technology and Innovation of Brazil.

\end{document}